\documentclass{amsart}%
\usepackage{amssymb}
\usepackage{cite}
\usepackage{amsfonts}
\usepackage{amsmath}
\usepackage{graphicx}%
\providecommand{\U}[1]{\protect\rule{.1in}{.1in}}
\newtheorem{theorem}{Theorem}
\theoremstyle{plain}

\newtheorem{definition}{Definition}
\newtheorem{example}{Example}

\newtheorem{lemma}{Lemma}

\numberwithin{equation}{section}
\begin{document}
\title[Distance Distribution Functions and Probabilistic Metric Spaces on C*-algebras]{Distance Distribution Functions and Probabilistic Metric Spaces on C*-algebras}
\author{Rasoul Abazari}
\email{rasoolabazari@gmail.com, rasoulabazari@iau.ac.ir}
\address{Department of Mathematics, Ardabil Branch, Islamic Azad University, Ardabil, Iran.}

\begin{abstract}
In this paper, by using the concept of positive elements of $C^*$-algebras instead of the real numbers $\mathbb{R}$, a generalization of distribution functions, with a particular focus on distance distribution functions has been introduced as a novel approach in statistical mathematics. Also, by applying this concept,  probabilistic metric space on $C^*$-algebras  is defined and some properties such as topology and fixed point theorems have been investigated. Finally, some examples and applications to integral equations are given.
\end{abstract}
\maketitle

\section{Introduction}
In probability theory, distribution functions play a fundamental role, therefore, their investigation and generalization are of significant importance. A {\em distribution function} is a nondecreasing and left continuouse function $F$ defined on $\mathbb{R}$, the set of real numbers, with $F(-\infty)=0$ and $F(\infty)=1$. If $F(0)=0$ then it is called {\em distance distribution function}.

In this paper, we present a generalization of the classical definition of a distribution function, with a particular focus on the distance distribution function. The motivation behind this generalization is to broaden the applicability of distribution functions in various mathematical and applied contexts. We begin by revisiting the foundational definitions and theorems, followed by the introduction of the generalized forms.

Classically, distribution functions are defined on the real numbers. By utilizing positive elements of a C*-algebra, we can redefine and generalize distribution functions within the framework of arbitrary C*-algebras, thereby enabling a broader range of results. 	Furthermore, by applying this generalized concept, a probabilistic metric space over C*-algebras is defined, and various properties such as topological structure and fixed point theorems are investigated. Finally, illustrative examples and applications to integral equations are provided.
\section{Main results}
Let $\mathcal{A}$ be an unital C*-algebra and $\mathcal{A}_h=\{a \in \mathcal{A} : a=a^*\}$. We call an element $a\in \mathcal{A}$ a positive element, denote it by $\theta \preceq a$, if $a\in \mathcal{A}_h$ and $\sigma(a)\subseteq [0,\infty)$, where $\theta$ is a zero element in $\mathcal{A}$ and $\sigma(a)$ is the spectrum of $a$. There is a natural partial ordering on $\mathcal{A}_h$ given by $a\preceq b$ if and only if $\theta\preceq b-a$. Let $\mathcal{A}_+$ denotes the set of all positive elements of $\mathcal{A}$ and $|a|=(a^*a)^{\frac{1}{2}}$. For a more detailed discussion on C*-algebras and their positive elements, see  \cite{Murphy,sakai2012c}.
A C*-algebra-valued metric space, as a generalization of metric spaces, was introduced by Ma et al., \cite{Z.H.Ma} by replacing the set of real numbers with the set of positive elements of a C*-algebra. 

\begin{definition}
	\cite{Z.H.Ma} Let $X$ be a nonempty set. Suppose the mapping $d:X\times X\longrightarrow \mathcal{A}_+$ satisfies;
	\begin{enumerate}
		\item $\theta\preceq d(x,y)$ for all $x,y\in X$ and $d(x,y)=\theta \Leftrightarrow x=y$,
		\item $d(x,y)=d(y,x)$ for all $ x,y\in X$,
		\item $d(x,y)\preceq d(x,z)+d(z,y)$ for all $x,y,z\in X$.
	\end{enumerate}   		
	Then $d$ is called a C*-algebra-valued metric on $X$ and $(X,\mathcal{A},d)$ is called a C*-algebra-valued metric space.
\end{definition}   
\begin{example}\label{Ex1}
	Let $X=\mathbb{R}^2$ and $\mathcal{A}=\mathcal{M}_2(\mathbb{R})$. Define the function $d: X\times X \to \mathcal{A}_+$ as follow: 
	$$d\left(a, b\right)= 
	\left(
	\begin{array}{cc}
		|x_1-x_2| & 0  \\
		0 & |y_1-y_2|  \\
	\end{array}
	\right)= diag\left(|x_1-x_2|, |y_1-y_2| \right) 
	$$
	for all $a=(x_1,y_1), b=(x_2,y_2)$ in $X$.\\
	Then, $d$ is a C*-algebra-valued metric on $X$. For this,
	clearly, $d(a,b)=\theta$ if and only if $a=b$, and $d(a,b)=d(b,a)$. To show the triangle property, let $c=(x_3, y_3) \in X$, then,
	\begin{align*}
		d(a,b)&=\left(
		\begin{array}{cc}
			|x_1-x_2| & 0  \\
			0 & |y_1-y_2|  \\
		\end{array}
		\right)\\
		&=\left(
		\begin{array}{cc}
			|x_1-x_3+x_3-x_2| & 0  \\
			0 & |y_1-y_3+y_3-y_2|  \\
		\end{array}
		\right)\\
		&\preceq \left(
		\begin{array}{cc}
			|x_1-x_3|+|x_3-x_2| & 0  \\
			0 & |y_1-y_3|+|y_3-y_2|  \\
		\end{array}
		\right)\\
		&=\left(
		\begin{array}{cc}
			|x_1-x_3| & 0  \\
			0 & |y_1-y_3|  \\
		\end{array}
		\right)+\left(
		\begin{array}{cc}
			|x_3-x_2| & 0  \\
			0 & |y_3-y_2|  \\
		\end{array}
		\right)\\
		&=d(a,c)+d(c,b).
	\end{align*}
	Therefore $d$ is a metric on $X$.
\end{example}

\begin{definition}
	\cite{Sh.Skalar} A {\em distance distribution function} (briefly a {\em d.d.f}) is a non-decreasing function $F$ defined on $\mathbb{R}_+$ that satisfies $F(0)=0$ and $F(\infty)=1$. and is left continuous on $(0,\infty)$. 
\end{definition}
Now, as a generalization, by replacing $\mathcal{A}_+$, the set of positive elements of a C*-algebra instead of $\mathbb{R}_+$, the definition of distance distribution functions on C*-algebra $\mathcal{A}$ is introduced as follow:
\begin{definition}\label{ddf}
	A {\em distance distribution function} (briefly, {\em d.d.f}) is a non-decreasing function $F$ defined on $\mathcal{A}_+$ that satisfies;
	$$F(\theta)=0 \quad , \ \lim_{n\to \infty}F(a_n)=1,$$
	for all strictly increasing sequence $\{a_n\}$ in $\mathcal{A}_+$
	and is left continuous on $\mathcal{A}_+$. 
\end{definition}
The set of all d.d.f 's on $\mathcal{A}$ will be denoted by $\Delta^+_{\mathcal{A}}$.

This study applies the concept of t-norms referenced in \cite{klement2013triangular, Schweizer}.
\begin{definition}\label{DT}
	A mapping  $T:[0,1]\times [0,1] \to [0,1]$ is called a {\em continuous t-norm}, if $T$ satisfies the following conditions:
	\begin{enumerate}
		\item $T$ is commutative and associative, i.e., $T(a,b)=T(b,a)$  and $T(a,T(b,c))=T(T(a,b),c)$, for all $a,b,c\in [0,1]$;
		\item $T$ is continuous;
		\item $T(a,1)=a$ for all $a\in [0,1]$;
		\item $T(a,b)\leq T(c,d)$ whenever $a\leq c$ and $b\leq d$ for all $a,b,c\in [0,1]$.
	\end{enumerate}
\end{definition}
\begin{definition}
	\cite{Schweizer} A Menger probabilistic metric space ($PM$-space) is a triple $(X,F,T)$, where $X$ is non-empty set, $T$ is a continuous $t$-norm and $F$ is a mapping from $X\times X\to \Delta^+$ satisfying the following conditions:\\
	($F_{(x,y)}$ denotes the value of $F$ at the pair $(x,y)$ )
	\begin{enumerate}
		\item $F_{(x,y)}(t)=1$ for all $x,y\in X$ and $t>0$ if and only if $x=y$;
		\item $F_{(x,y)}(t)=F_{(y,x)}(t)$;
		\item $F_{(x,y)}(t+s)\geq T\left(F_{(x,z)}(t), F_{(z,y)}(s)\right)$ for all $x,y,z\in X$ and $t,s\geq0$.
	\end{enumerate}
\end{definition}
Further details on probabilistic metric spaces can be found in \cite{huang2023fixed,NARUKAWA2023108528}.

Now, we introduce the definition of Menger probabilistic metric space on a C*-algebra. 
\begin{definition}\label{PM-def}
	A Menger probabilistic metric space on a C*-algebra $\mathcal{A}$ ($PM^*$-space), is a quadruple $(X,\mathcal{A},\mathcal{F},T)$, where $X$ is non-empty set, $T$ is a continuous $t$-norm and $\mathcal{F}:X\times X\to \Delta^+_{\mathcal{A}}$ is a mapping that satisfies the following conditions:\\
	($\mathcal{F}_{p,q}$ denotes the value of $\mathcal{F}$ at the pair $(p,q)$ )
	\begin{enumerate}
		\item $\mathcal{F}_{p,q}(t)=1$ for all $p,q\in X$ and $t\succ0$ if and only if $p=q$;
		\item $\mathcal{F}_{p,q}(t)=\mathcal{F}_{q,p}(t)$;
		\item $\mathcal{F}_{p,q}(t+s)\geq T\left(\mathcal{F}_{p,r}(t), \mathcal{F}_{r,q}(s)\right)$ for all $p,q,r\in X$, $s,t \in \mathcal{A}_+$.
	\end{enumerate}
\end{definition}
Below, we will provide a few examples.
\begin{example}\label{Ex2}
	Let $(X, \mathcal{A}, d)$ be a C*-algebra-valued metric space as in Example \ref{Ex1} and $\mathcal{H}_0 , \mathcal{D}: \mathcal{A}_+ \to [0, 1]$, be functions as follow;\\
	\[
	\begin{aligned}
		\mathcal{H}_0(C) &= 
		\begin{cases}
			0 & ,\  C=\theta \\
			1 & ,\  C \succ \theta 
		\end{cases}
		& \ , \  &
		\mathcal{D}(C) &= 
		\begin{cases}
			0 & ,\  C=\theta \\
			1-e^{-tr(C)} & ,\  C\succ\theta 
		\end{cases}
	\end{aligned}
	\]
	that, $tr(C)$ is {\em trace} of matrix $C$. Obviously, $\mathcal{H}_0$ and $\mathcal{D}$ belong to $\Delta^+_{\mathcal{A}}$.\\ Define the function $\mathcal{F}: X\times X \to \Delta^+_{\mathcal{A}} $ as follow;
	\[
	\mathcal{F}_{a,b}(C) = 
	\begin{cases}
		\mathcal{H}_0(C) & ,\  a=b \\
		\mathcal{D}(\frac{1}{tr\left(d(a,b)\right)}C) & ,\ a\neq b
	\end{cases},
	\]
	
	for all $C\in \mathcal{A}_+$.\\
	We show that $\mathcal{F}_{a,b}$ is a distance distribution function on C*-algebra $\mathcal{A}$. For this, clearly $\mathcal{F}_{a,b}$ satisfies properties (i) and (ii) in the definition \ref{PM-def}. to prove property (iii), let $A$ and $B$ be in $\mathcal{A}_+$ and $T=T_{\min}$ as $t$-norm. We have to show that;
	$$\mathcal{D}\left( \frac{1}{tr\left(d(a,b)\right)}(A+B)\right) \geq \min \left\{\mathcal{D}\left( \frac{1}{tr\left(d(a,c)\right)}A\right) , \mathcal{D}\left( \frac{1}{tr\left(d(c,b)\right)}B \right)\right\},  $$
	for $a,b,c\in X$.\\
	Since $d(a,b)\preceq d(a,c)+d(c,b)$, so
	$$tr\left( d(a,b)\right) \leq tr\left( d(a,c)\right) +tr\left( d(c,b)\right) $$
	$$\frac{1}{tr\left(d(a,b)\right)}(A+B)\succeq \frac{1}{tr\left( d(a,c)\right) +tr\left( d(c,b)\right)}(A+B)$$
	On the other hands;
	\begin{align*}
		\max\left\{\frac{1}{tr\left(d(a,c)\right)}A, \frac{1}{tr\left(d(c,b)\right)}B\right\} &\succeq \frac{1}{tr\left( d(a,c)\right) +tr\left( d(c,b)\right)}(A+B)\\
		&\succeq \min \left\{\frac{1}{tr\left(d(a,c)\right)}A, \frac{1}{tr\left(d(c,b)\right)}B\right\}
	\end{align*}
	So, 
	$$\frac{1}{tr\left(d(a,b)\right)}(A+B)\succeq \min \left\{\frac{1}{tr\left(d(a,c)\right)}A, \frac{1}{tr\left(d(c,b)\right)}B\right\}.$$
	Since $\mathcal{D}$ is non-decreasing, hence;
	$$\mathcal{D}\left( \frac{1}{tr\left(d(a,b)\right)}(A+B)\right) \geq \min \left\{\mathcal{D}\left( \frac{1}{tr\left(d(a,c)\right)}A\right) , \mathcal{D}\left( \frac{1}{tr\left(d(c,b)\right)}B \right)\right\}, $$
	and therefore;
	$$\mathcal{F}_{a,b}(A+B)\geq T_{\min}\left(\mathcal{F}_{a,c}(A),\mathcal{F}_{c,b}(B) \right).$$
	So, $(X, \mathcal{A}, \mathcal{F}, T_{\min})$ is a probabilistic metric space.
\end{example}

\begin{example}\label{Ex3}
	Let $X, \mathcal{A}, T$ and $d$ be as in previous example. Define $\mathcal{F}:X\times X \to \Delta^+_{\mathcal{A}}$ as follow;
	$$\mathcal{F}_{q,q}(C)=\frac{tr(C)}{tr\left( C+d(p,q)\right) },$$
	for all $C\in \mathcal{A}_+$.
	Clearly, $\mathcal{F}$ satisfies properties (i) and (ii) in the definition \ref{PM-def}. To show the property (iii), Let $A, B \in \mathcal{A}$ and $p, q, r \in X$. We have;
	\begin{align*}
		\mathcal{F}_{p,q}(A+B)&=\frac{tr(A+B)}{tr\left( A+B+d(p,q)\right)}\\
		&\geq \frac{tr(A+B)}{tr\left( A+B+d(p,r)+d(r,q)\right)}\\
		&\geq \min\left\{\frac{tr(A)}{tr\left( A+d(p,r)\right)}, \frac{tr(B)}{tr\left( B+d(r,q)\right)}\right\}\\
		&= \min\left\{\mathcal{F}_{p,r}(A), \mathcal{F}_{r,q}(B)\right\}.
	\end{align*}
	Therefore, $(X, \mathcal{A}, \mathcal{F}, T_{\min})$ is a probabilistic metric space.
\end{example}
\section{Topology and fixed point theorems}
In this section, a topology for probabilistic metric spaces on a C-algebra is introduced, following the framework presented in \cite{Sh.Skalar}, and corresponding fixed point theorems are investigated.
\begin{definition}
	Let $(X,\mathcal{A},\mathcal{F},T)$ be a $PM^*$-space. For $p\in X$ , $t\succeq \theta$ and $\lambda\in (0,1)$, the {\em strong $(t,\lambda)$-neighborhood} of $p$ is defined by the set 
	$$\mathcal{N}_p(t,\lambda)=\{q \in X : \mathcal{F}_{p,q}(t)>1-\lambda\}.$$
\end{definition}
$\mathcal{N}_p(t,\lambda)$ is the set of all points $q$ in $X$ for which the probability of the distance from $p$ to $q$ being less than $t$ is greater than $1-\lambda$. Observe that this neighborhood of a point in an PM*-space depends on two parameters.

\begin{lemma}
	If $t_1\preceq t_2$ and $\lambda_1 \leq \lambda_2$, then $\mathcal{N}_p(t_1,\lambda_1)\subseteq \mathcal{N}_p(t_2,\lambda_2)$.
\end{lemma}
\begin{proof}
	Let $q\in \mathcal{N}_p(t_1,\lambda_1)$ so $\mathcal{F}_{p,q}(t_1)>1-\lambda_1$, since $\mathcal{F}_{p,q}$ is non-decreasing and $\lambda_1 \leq \lambda_2$, therefore,
	$$\mathcal{F}_{p,q}(t_2)\geq \mathcal{F}_{p,q}(t_1)>1-\lambda_1\geq 1-\lambda_2,$$
	and hence,  $q\in \mathcal{N}_p(t_2,\lambda_2)$.
\end{proof}

\begin{theorem}
	Let $(X,\mathcal{A},\mathcal{F},T)$ be a Menger PM*-space with continuous $t$-norm $T$, if $\mathcal{A}$ be totally ordered C*-algebra, then $(X,\mathcal{A},\mathcal{F},T)$ is a Hausdorff space in the topology induced by the family of $(t,\lambda)$-neighborhoods $\{\mathcal{N}_p\}$.
\end{theorem}
\begin{proof}
	The proof is similar to Theorem (7.2) in \cite{Sh.Skalar}. 
\end{proof}

\begin{definition}\label{def-conv}
	Let $\{p_n \}$ be a sequence in a $PM^*$-space $X$. Then;
	\begin{enumerate}
		\item $\{p_n \}$ is said to be a Cauchy sequence, if, for every $t\succ \theta$ and  $\lambda > 0$, there exists an integer $M_{t, \lambda}$ , such that $\mathcal{F}_{p_n,p_m}(t) > 1-\lambda$, for every $n,m> M_{t, \lambda}$.\\
		\item $\{p_n \}$ is said to converge to a point $p$ in $X$ (and we write $p_n \to p$) if, for every $t\succ \theta$ and every $\lambda > 0$, there exists an integer $M_{t, \lambda}$ , such that $p_n \in \mathcal{N}_p (t, \lambda)$, i.e., $\mathcal{F}_{p,p_n}(t) > 1-\lambda$, whenever $n > M_{t, \lambda}$ .
	\end{enumerate} 
\end{definition}
\begin{definition}
	A $PM^*$-space is said to be {\em complete}, if every Cauchy sequence is convergent.
\end{definition}

\begin{theorem}
	If a sequence in a $PM^*$-space be convergent, then the convergence is unique.
\end{theorem}
\begin{proof}
	Let $\{p_n\}$ be a sequence in a $PM^*$-space $X$, such that converges to two point $p, q \in X$. we have;
	$$\mathcal{F}_{p,q}(t)\geq T\left(\mathcal{F}_{p,p_n}(t/2
	),\mathcal{F}_{q,p_n}(t/2) \right),$$
	for every $t\in \mathcal{A}$. Since
	$$\lim_{n\to \infty}\mathcal{F}_{p,p_n}(t/2
	)=\lim_{n\to \infty}\mathcal{F}_{q,p_n}(t/2
	)=1,$$ so $\mathcal{F}_{p,q}(t)=1$ and therefore $p=q$.
\end{proof}

\begin{lemma}
	$p_n\to p$ if and only if $\mathcal{F}_{pp_n} \to \mathcal{F}_{p,p} = \mathcal{H}_0$.
\end{lemma}

\begin{proof}
	By definition \ref{def-conv}, the proof is straightforward.  
\end{proof}

\begin{definition}
	Let  $(X,\mathcal{A},\mathcal{F},T)$ be a PM-space. A mapping $f:X\to X$ is said to be $a$-contraction, if there exists a constant $a\in \mathcal{A}$ with $a\succ e$ such that
	$$\mathcal{F}_{fp,fq}(t)\geq \mathcal{F}_{p,q}(a^*ta)$$
	for all $p, q \in X$ and $t\in \mathcal{A}$.
\end{definition}

\begin{definition}
	\cite{Hadzic}  Let $T$ be a $t$-norm and $T^n : [0, 1] \to [0, 1],  (n\in \mathbb{N})$ be defined as follow:
	$$T^1(a)=T(a,a), \quad T^2(a)=T(T^1(a),a),\quad ... ,\quad T^n(a)=T(T^{n-1}(a),a)$$
	for all $n\in \mathbb{N}$ and $a\in [0,1]$. The $t$-norm $T$ is said to be of {\em Had?i?-type}, if $T$ is continuous and the family $\{T^n(a)\}_{n\in \mathbb{N}}$ is equicontinuous at $a=1$. that is, for any $\epsilon \in (0,1)$, there exists $\delta\in (0,1)$ such that
	$$a>1-\delta \quad \Rightarrow \quad T^n(a)> 1-\epsilon$$
	for all $n\geq 1.$ 
\end{definition}
A trivial example of t-norm of Had?i?-type is $T = T_{\min}(a, b) = min\{a, b\}.$ 

In the following theorem, we examine the fixed points of functions on $X$.
\begin{theorem}\label{FB1}
	If  $(X,\mathcal{A},\mathcal{F},T)$ be a complete Menger probabilistic metric space with $T$ Had?i?-type and $f: X \to X$ be a $a$-contraction function, then, for any $p_0\in X$, the sequence $\{f^np_{0}\}$
	converges to a unique fixed point of $f$. 
\end{theorem}
\begin{proof}
	Let $t \in \mathcal{A}$, since $a\succ e$ so $a^{-1}\prec e$, on the other hands, 
	$$a^{-1}ta^{-1}=(t^{\frac{1}{2}}a^{-1})^*(t^{\frac{1}{2}}a^{-1})=| t^{\frac{1}{2}}a^{-1}|^2\preceq |t^{\frac{1}{2}}|^2|a^{-1}|^2\preceq ||a^{-1}||^2(t^{\frac{1}{2}}t^{\frac{1}{2}})=||a^{-1}||^2t.$$
	Since $||a^{-1}||<1$, hence $a^{-1}ta^{-1}\prec t.$\\
	Let $p_0\in X$ and set $p_{n+1}=fp_n=f^{n+1}p_0$, for $n= 1,2, ...$ . Then we have;
	\begin{align*}
		\mathcal{F}_{p_{n+1},p_{n}}(t)&=\mathcal{F}_{fp_{n},fp_{n-1}}(t)\\
		&\geq\mathcal{F}_{p_{n},p_{n-1}}(a^*ta)\\
		&\geq\mathcal{F}_{p_{n-1},p_{n-2}}\left( (a^*)^2t(a)^2\right)\\
		&\geq \cdots\\
		&\geq\mathcal{F}_{p_{1},p_{0}}\left( (a^*)^nt(a)^n\right)\\
	\end{align*}
	Since $\{ (a^*)^nt(a)^n\}$ is a strictly increasing sequence in $\mathcal{A}$, from Definition \ref{ddf}, we have, $\mathcal{F}_{p_{1},p_{0}}\left( (a^*)^nt(a)^n\right) \to 1$, as $n\to \infty$, so, 
	\begin{equation*}
		\lim_{n\to \infty} \mathcal{F}_{p_{n+1},p_{n}}(t)=1.
	\end{equation*}
	
	By the monotonicity of $T$, it follows that;
	\begin{align*}
		&\mathcal{F}_{p_{n},p_{n+k}}(t)=T\left( \mathcal{F}_{p_{n},p_{n+1}}(t-a^{-1}ta^{-1}), \mathcal{F}_{p_{n+1},p_{n+k}}(a^{-1}ta^{-1})\right) \\
		&=T\left( \mathcal{F}_{p_{n},p_{n+1}}(t-a^{-1}ta^{-1}), \mathcal{F}_{fp_{n},fp_{n+k-1}}(a^{-1}ta^{-1})\right) \\
		& \succeq T\left( \mathcal{F}_{p_{n},p_{n+1}}(t-a^{-1}ta^{-1}), \mathcal{F}_{p_{n},p_{n+k-1}}(t)\right) \\
		&\cdots\\
		& \succeq T^k\left( \mathcal{F}_{p_{n},p_{n+1}}(t-a^{-1}ta^{-1})\right).
	\end{align*}
	Hence;
	$$\mathcal{F}_{p_{n},p_{n+k}}(t)\succeq T^k\left( \mathcal{F}_{p_{n},p_{n+1}}(t-a^{-1}ta^{-1})\right).$$
	Since $\{T^k: k\geq0\}$ is equicontinuous at 1 and $T^k(1)=1$, so, for any $\epsilon>0$, there exists $\delta \in (0,1)$ such that, for any $\alpha> 1-\delta$,
	$$T^k(\alpha)>1-\epsilon,$$
	for all $k\geq 1$. Since $\lim_{n\to\infty}\mathcal{F}_{p_{n},p_{n+1}}(t-a^{-1}ta^{-1})=1$, hence there exists $n_0\in \mathbb{N}$, such that for all $n\geq n_0$, 
	$\mathcal{F}_{p_{n},p_{n+1}}(t-a^{-1}ta^{-1})>1-\delta$, hence, $\mathcal{F}_{p_{n},p_{n+k}}(t)>1-\epsilon$, and therefore,
	$$\lim_{n,m\to \infty}\mathcal{F}_{p_{n},p_{m}}(t)=1.$$
	So, we conclude that, $\{p_n\}$ is Cauchy sequence. By the completeness of $PM^*$-space $X$, there exists $p\in X$ such that, 
	$$\lim_{n\to \infty}\mathcal{F}_{p_{n},p}(t)=1.$$
	On the other hands,
	$$\mathcal{F}_{f{p_n},fp }(t)\geq \mathcal{F}_{p_n,p}(ata).$$
	Since $p_n\to p$ and $f{p_n}\to p$, hence,
	$$\mathcal{F}_{fp,p}(t)=1,$$
	for any $t>0$. So $fp=p$.\\
	To show the uniqueness of $p$, let $q$ be the another fixed point of $f$, we have, 
	$$\mathcal{F}_{p,q}(t)=\mathcal{F}_{fp,fq}(t)\geq\mathcal{F}_{p,q}(ata)= \mathcal{F}_{fp,fq}(ata)\geq \cdots \geq \mathcal{F}_{p,q}(a^nta^n)$$
	Since $\lim_{n\to \infty}\mathcal{F}_{p,q}(a^nta^n)=1$, so, $\mathcal{F}_{p,q}(t)=1$. Hence, $p=q$ and therefore $f$ has a unique fixed point in $X$.
\end{proof}
\section{Examples and Applications}
This section presents several examples and applications to illustrate the theoretical results.
\begin{example}
	Suppose $(X, \mathcal{A}, \mathcal{F}, T_{\min})$ be as in Example \ref{Ex2}, and let $0<\lambda_1, \lambda_2<1$ and $\lambda=\max\{\lambda_1, \lambda_2\}$. Define a mapping $f:X \to X$ by $f(x,y)=(\alpha_1+\lambda_1x, \alpha_2+\lambda_2y)$ for $\alpha_1, \alpha_2\in \mathbb{R}$ and $(x,y)\in X$. Let $p=(x_1,y_1)$ and $q=(x_2,y_2)$ be in $X$ such that $p\neq q$, we have;
	\begin{align*}
		d(fp,fq)&=\left(
		\begin{array}{cc}
			\lambda_1|x_1-x_2| & 0  \\
			0 & \lambda_2|y_1-y_2|  \\
		\end{array}
		\right)\\
		&\preceq \lambda \left(
		\begin{array}{cc}
			|x_1-x_2| & 0  \\
			0 & |y_1-y_2|  \\
		\end{array}
		\right)\\
		&=\lambda d(p,q).     		
	\end{align*} 
	So, 
	\begin{align*}
		\mathcal{F}_{fp,fq}(C)&=1-e^{-\frac{tr(C)}{tr\left( d(fp,fq)\right)}}\\
		&\geq 1-e^{-\frac{tr(C)}{\lambda tr\left( d(p,q)\right)}}\\
		&= 1-e^{-\frac{tr\left( \varLambda C\varLambda\right) }{tr\left( d(p,q)\right)}}\\
		&=\mathcal{F}_{p,q}\left( \varLambda C\varLambda\right),
	\end{align*}
	for every $C\succ \theta$, where, $\varLambda\in \mathcal{A}$ as bellow;
	$$\varLambda=\left(
	\begin{array}{cc}
		\frac{1}{\sqrt{\lambda}} & 0  \\
		0 & \frac{1}{\sqrt{\lambda}}  \\
	\end{array}
	\right).$$
	Therefore, $f$ is a $\varLambda$-contraction and  hence, by Theorem \ref{FB1}, $f$ has a fixed point in $X$. Indeed, fixed point is $p=\left(\frac{\alpha_1}{1-\lambda_1}, \frac{\alpha_2}{1-\lambda_2} \right)$.
\end{example}
\begin{example}
	Let $d, f,$ and $\varLambda$ be as in the previous example, the result also is true for PM*-space discussed in Example \ref{Ex3}.  
\end{example}
\begin{example}
	Consider the system consisting of two Fredholm integral equations of the second
	kind in two unknowns:
	$$\phi_1(x)=g_1(x)+\int_{a}^{b}K_{11}(x,t)\phi_1(t)+K_{12}(x,t)\phi_2(t)dt$$
	$$\phi_2(x)=g_2(x)+\int_{a}^{b}K_{21}(x,t)\phi_1(t)+K_{22}(x,t)\phi_2(t)dt$$
	where $K_{ij}(x,t)$ are given kernel functions, and $f_i(x)$ are given functions for $x\in [a,b]$. The above system can be express as follow;
	\begin{align*}
		\left(
		\begin{array}{c}
			\phi_1(x)   \\
			\phi_2(x)  \\
		\end{array}
		\right)=
		\left(
		\begin{array}{c}
			g_1(x)   \\
			g_2(x)  \\
		\end{array}
		\right)+\int_{a}^{b}
		\left(
		\begin{array}{cc}
			K_{11}(x,t) & K_{12}(x,t)  \\
			K_{21}(x,t) & K_{22}(x,t)  \\
		\end{array}
		\right)
		\left(
		\begin{array}{c}
			\phi_1(x)   \\
			\phi_2(x)  \\
		\end{array}
		\right)dt
	\end{align*}
	Let $\|K\|=\max\left\{\|K_{11}\|+\|K_{21}\|, \|K_{12}\|+\|K_{22}\|\right\}$
	, $X=C([a,b],\mathbb{R})\times C([a,b],\mathbb{R})$ and $\mathcal{A}=\mathcal{M}_{2}(\mathbb{R})$, define function $f: X \to X$ as follow;
	$$f(u,v)=\left(f_1(u,v), f_2(u,v)\right)$$
	for all $(u,v)\in X$, where;
	$$f_1(u,v)(x)=g_1(x)+\int_{a}^{b}K_{11}(x,t)u(t)+K_{12}(x,t)v(t)dt,$$  $$f_2(u,v)(x)=g_2(x)+\int_{a}^{b}K_{21}(x,t)u(t)+K_{22}(x,t)v(t)dt.$$
	We have;
	\begin{align*}
		\|f_1(u_1,v_1)-f_1(u_2,v_2)\|&\leq \int_{a}^{b}\|K_{11}\|\|u_1-u_2\|+\|K_{12}\|\|v_1-v_2\|dt\\
		&\leq (b-a)\|K\|\left(\|u_1-u_2\|+\|v_1-v_2\|\right)
	\end{align*}
	Similarly,
	\begin{align*}
		\|f_2(u_1,v_1)-f_2(u_2,v_2)\|\leq (b-a)\|K\|\left(\|u_1-u_2\|+\|v_1-v_2\|\right)
	\end{align*}
	
	Consider the metric $d: X\times X \to \mathcal{A}_+$ as follow;
	\begin{align*}
		d((u_1, v_1),(u_2, v_2))&=\left(
		\begin{array}{cc}
			\|u_1-u_2\|+\|v_1-v_2\| & 0  \\
			0 & \|u_1-u_2\|+\|v_1-v_2\|  \\
		\end{array}
		\right).
	\end{align*}
	We have,
	\begin{align*}
		&d\left(f(\phi_1, \phi_2),f(\psi_1,\psi_2) \right)\\
		&=\left(
		\begin{array}{cc}
			\|f_1(\phi_1,\phi_1)-f_1(\psi_2,\psi_2)\|+\|f_2(\phi_1,\phi_1)-f_2(\psi_2,\psi_2)\| \qquad 0  \\
			0 \qquad \|f_1(\phi_1,\phi_1)-f_1(\psi_2,\psi_2)\|+\|f_2(\phi_1,\phi_1)-f_2(\psi_2,\psi_2)\|\\
		\end{array}
		\right)	\\
		&\leq 2(b-a)\|K\|\left(
		\begin{array}{cc}
			\|\phi_1-\phi_2\|+\|\psi_1-\psi_2\| & 0  \\
			0 & \|\phi_1-\phi_2\|+\|\psi_1-\psi_2\|  \\
		\end{array}
		\right)\\ 
		&\leq  2(b-a)\|K\|d\left((\phi_1, \phi_2),(\psi_1,\psi_2) \right) 
	\end{align*}
	Hence,
	$$d\left(f(\phi_1, \phi_2),f(\psi_1,\psi_2) \right)\leq  2(b-a)\|K\|d\left((\phi_1, \phi_2),(\psi_1,\psi_2) \right)$$
	If we suppose $\|K\|\leq \frac{1}{2(b-a)}$, then we have;
	\begin{align*}
		\mathcal{F}_{f(\phi_1, \phi_2),f(\psi_1,\psi_2)}(C)&=1-e^{\frac{tr(C)}{tr \left(d(f(\phi_1, \phi_2),f(\psi_1,\psi_2))\right)}}\\
		&\geq 1-e^{\frac{tr(C)}{2(b-a)\|K\| tr\left(d((\phi_1, \phi_2),(\psi_1,\psi_2))\right)}}\\
		&=1-e^{\frac{\frac{1}{2(b-a)\|K\|}tr(C)}{tr\left(d((\phi_1, \phi_2),(\psi_1,\psi_2))\right)}}\\
		&=1-e^{\frac{tr(DCD)}{tr\left(d((\phi_1, \phi_2),(\psi_1,\psi_2))\right)}}\\
		&=\mathcal{F}_{(\phi_1, \phi_2),(\psi_1,\psi_2)}(DCD).
	\end{align*}
	for every $C\succ \theta$, where;
	\begin{align*}
		D=\left(
		\begin{array}{cc}
			\frac{1}{\sqrt{2(b-a)\|K\|}} & 0  \\
			0 & \frac{1}{\sqrt{2(b-a)\|K\|}}\\
		\end{array}
		\right)	\\
	\end{align*}
	Since $D\succ I$, therefore $f$ is a $D$-contraction and by Theorem \ref{FB1}, $f$ has a unique fixed point.
\end{example}

\section{Conclusion}

In this study, we have extended the classical concept of distribution functions by redefining them within the framework of arbitrary C*-algebras, with particular emphasis on distance distribution functions. This generalization allows for a more flexible and abstract formulation, enabling their use in a broader range of mathematical contexts, including operator theory and functional analysis.

By employing positive elements in C*-algebras, we established a novel form of probabilistic metric spaces that naturally align with the generalized distribution functions. Within this setting, we examined key topological properties and proved fixed point theorems that generalize existing results in probabilistic metric spaces. These developments not only deepen the theoretical understanding of distribution functions in noncommutative settings but also pave the way for new applications.

Illustrative examples and applications to integral equations were also presented, demonstrating the utility and potential of this generalized framework in addressing real-world mathematical problems. Future research may explore further extensions of this approach, including its interactions with quantum probability and noncommutative geometry.


\begin{thebibliography}{99}                                                                                               %

	\bibitem{Hadzic} {O. Hadzic, A fixed point theorem in Menger space, {\em Publ. Ins. Math. (Belgr)},  20, (1979) 107-112.}

\bibitem{huang2023fixed} {H. Huang, T. Dosenovic, D. Rakic , S. Radenovic,
	Fixed Point Results in Generalized Menger Probabilistic Metric Spaces with Applications to Decomposable Measures,
	{\em Axioms}, Vol. 12, No. 7 2014, (2023) 660.}

\bibitem{klement2013triangular} {E.P. Klement, R. Mesiar, E. Pap,
	Triangular norms,
	{\em Springer Science \& Business Media}, 8, (2013).}


\bibitem{Z.H.Ma} {Z. H. Ma, L.N. Jiang, H.K. Sun,
	C*-algebra-valued metric spaces and related fixed point theorems,
	{\em Fixed Point Theory and
		Applications},  2014, (2014) 206.}

\bibitem{Murphy} { GJ. Murphy, C*-Algebras and Operator Theory,
	{\em  Academic Press, London}, (1990).}	

\bibitem{NARUKAWA2023108528} {Y. Narukawa, M. Taha, V. Torra,
	On the definition of probabilistic metric spaces by means of fuzzy measures,
	{\em Fuzzy Sets and Systems},   Vol, 465, (2023) 108528.}

\bibitem{sakai2012c} {S. Sakai,
	C*-algebras and W*-algebras,
	{\em Springer Science \& Business Media}, (2012).}

\bibitem{Schweizer} B. Schweizer, A. Sklar \emph{Probabilistic metric space}, Elsevier North Holand, New York, (1983).	

\bibitem{Sh.Skalar} {B. Schweizer, A. Sklar,
	Statistical Metric Space,
	{\em Pacific Journal of Mathematics},   Vol, 10 No. 1, (1960) 313-334.}
\end{thebibliography}
\end{document}